\documentclass[12pt,leqno,amsfonts,amscd]{amsart}
\usepackage{epsfig}
\usepackage{amsmath}
\usepackage{amsfonts}
\usepackage{amssymb}
\usepackage{color}
\usepackage{hyperref}

\newtheorem{theorem}{Theorem}[section]
\newtheorem{prop}[theorem]{Proposition}
\newtheorem{cor}[theorem]{Corollary}
\theoremstyle{definition}

\theoremstyle{remark}

\newtheorem{remark}[theorem]{Remark}

\numberwithin{equation}{section}

\newcommand{\NN}{{\mathbb N}}
\newcommand{\QQ}{{\mathbb Q}}
\newcommand{\RR}{{\mathbb R}}
\newcommand{\CC}{{\mathbb C}}

\newcommand{\out}[1]{\ }

\let\PSH=\psh

\let\cal=\mathcal
\renewcommand{\phi}{\varphi}

\begin{document}

\title[An equality on Monge-Amp\`re measures]{An equality on Monge-Ampere measures}

\author[M. El Kadiri]{Mohamed El Kadiri}
\address{University of Mohammed V
\\Department of Mathematics,
\\Faculty of Science
\\P.B. 1014, Rabat
\\Morocco}
\email{elkadiri@fsr.ac.ma}




\begin{abstract} Let $u$ and $v$ be two plurisubharmonic functions in the
domain of definition of the Monge-Amp\`ere operator on
a domain $\Omega\subset \CC^n$. We prove that if $u=v$ on a plurifinely
open set $U\subset \Omega$ that is Borel measurable, then $(dd^cu)^n|_U=(dd^cv)^n|_U$.
This result was proved by Bedford and Taylor in \cite{BT} in the case where
$u$ and $v$ are locally bounded, and by El Kadiri and Wiegerinck \cite{EKW} when
$u$ and $v$ are finite, and by Hai and Hiep in \cite{HH} when $U$ is of the form
$U=\bigcup_{j=1}^m\{\varphi_j>\psi_j\}$, where $\varphi_j$, $\psi_j$, $j=1,...,m$, are
plurisubharmonic functions on $\Omega$.
\end{abstract}
\maketitle

\section{Introduction}

Let $u$ and $v$ be two locally bounded plurisubharmonic (psh in abbreviated form) functions on a domain $\Omega\subset \CC^n$
and $U$ a plurifinely open set $\subset \Omega$. In  \cite{BT} Bedford and Taylor
proved that if $u=v$ on $U$  then the restrictions of the Monge-Amp\`ere measures $(dd^cu)^n$
and $(dd^cv)^n$ to $U$ are equal (see \cite[Corollary 4.3]{BT}). El Kadiri and Wiegerinck
showed in \cite{EKW} that this result is true when $u$ and $v$ are two psh functions in the general domain
of definition of the Monge-Amp\`ere operator, finite on $U$,
and they extended it to the setting of $\cal F$-plurisubharmonic functions on $U$.
For the latter notion see \cite{EK} and \cite{EKFW}. Hai and Hiep  extended this result to the large Cegrell class of
plurisubharmonic functions on a hyperconvex domain $\Omega$ and a plurifine
open subset $U$ of $\Omega$ of the form $U=\bigcup_{j=1}^m\{\varphi>\psi_j\}$,
where $\varphi_j$ and  $\psi_j$ ($1\leq i, j\leq m$) are plurisubharmonic functions in the
Cegrell class $\cal E(\Omega)$ (see \cite[Theorem 1.1]{HH}).
Recall here that the plurifine topology on an open set in $\CC^n$ is
the coarsest topology on $\Omega$ that makes all plurisubharmonic functions
on $\Omega$ continuous. The plurifine topology has been investigated by many authors,
namely, Bedford and Taylor, El Marzguioui, Fuglede and Wiegerinck, see \cite{BT},
\cite{EK}, \cite{EKFW} and \cite{EMW}.

Our main purpose in this paper is to extend the above results
to any plurifinely open set that is Borel measurable. Indeed, we shall prove that if $u$ and
$v$ are two plusubharmonic functions in the Cegrell class
$\cal E(\Omega)$, where $\Omega$ is a bounded hyperconvex domain of $\CC^n$,
are equal on  a plurifinely open set $U\subset \Omega$ that is Borel
measurable, then $(dd^cu)^n|_U=(dd^cv)^n|_U$. This implies that this result is true when $\Omega$
is a general open subset of $\CC^n$ and $\cal E(\Omega)$ is replaced by the general domain of
definition $\cal D$ of the Monge-Amp\`ere operator on $\Omega$, in the sense of Blocki \cite{Bl}.

\section{The cegrell classes}
Let $\Omega$ be a bounded hyperconvex domain in $\CC^n$. From \cite{Ce1} and \cite{Ce2} we recall
the following subclasses of $\PSH_-(\Omega)$, the cone of nonpositive plurisubharmonic functions
on $\Omega$:

$${\cal E}_0={\cal E}_0(\Omega)=\{\varphi\in \PSH_-(\Omega): \lim_{z\to \partial \Omega}\varphi(z)=0, \ \int_\Omega (dd^c\varphi)^n<\infty\},$$

$${\cal F}={\cal F}(\Omega)=\{\varphi\in \PSH_-(\Omega): \exists \ \cal E_0\ni \varphi_j\searrow \varphi, \
\sup_j\int_\Omega (dd^c\varphi_j)^n<\infty\},$$

and
\begin{equation*}
\begin{split}
\cal E=\cal E(\Omega)=\{\varphi \in \PSH_-(\Omega): \forall z_0\in \Omega, \exists
\text{ a neighborhood } \omega \ni z_0,&\\
{\cal E}_0\ni \varphi_j \searrow \varphi \text{ on } \omega, \ \sup_j \int_\Omega(dd^c\varphi_j)^n<\infty \}.
\end{split}
\end{equation*}

As in \cite{Ce2}, we note that if $u\in \PSH_-(\Omega)$ then $u\in \cal E(\Omega)$ if and only
if for every $\omega\Subset \Omega$, there is $v\in \cal F(\Omega)$ such that $v\ge u$ and
$v=u$ on $\omega$. On the other hand we have $\PSH_-(\Omega)\cap L^\infty_{loc}(\Omega)\subset \cal E(\Omega)$.
The classical Monge-Amp\`ere operator on $\PSH_-(\Omega)\cap L^\infty_{loc}(\Omega)$ can be extended uniquely to
the class $\cal E(\Omega)$, the extended operator is still denoted by $(dd^c\cdot)^n$.
According to Theorem 4.5 from \cite{Ce1}, the class $\cal E$ is the biggest class
$\cal K\subset \PSH_-(\Omega)$ satisfying the following conditions:

(1) If $u\in \cal K$, $v\in \PSH_-(\Omega)$ then $\max(u,v)\in \cal K$.

(2) If $u\in \cal K$, $\varphi_j\in \PSH_-(\Omega)\cap L^\infty_{loc}(\Omega)$, $\varphi_j\searrow u$, $j\to +\infty$,
then $((dd^c\varphi_j)^n)$ is weak*-convergent.

We also recall, following Blocki, cf \cite{Bl}, that the general domain of definition $\cal D$ of
the Monge-Amp\`ere operator on a domain $\Omega$ of $\CC^n$ consists of plurisubharmonic functions $u$
on $\Omega$ for which there is a nonnegative (Radon) measure $\mu$ on $\Omega$ such that for any decreasing sequence
$(u_j)$ of locally bounded plurisubharmonic functions on $\Omega$, the sequence of measures $(dd^cu_j)^n$ is weakly*-convergent
to $\mu$. The measure $\mu$ is denoted by $(dd^cu)^n$ and called the Monge-Amp\`ere of (or associated with) $u$.
When $\Omega$ is bounded and hyperconvex then $\cal D\cap\PSH_-(\Omega)$ coincides with the class $\cal F=\cal F(\Omega)$, cf. \cite{Bl}.

\section{Equality between the Monge-Amp\`ere measures on plurifine open sets}

The following theorem was proved by Hai and Hiep in \cite{HH}:

\begin{theorem}[{\cite[Theorem 1.1]{HH}}]\label{thm1.1}
Let $\Omega$ be a bounded hyperconvex domain in $\CC^n$ and
$\varphi_1,...,\varphi_m$, $\psi_1,...,\psi_m$ are plurusubharmonic functions
on $\Omega$. Let $U=\{\varphi_1>\psi_1\}\cap ...\cap \{\phi_m > \psi_m\}$. Assume that
$u,v \in \cal E(\Omega)$. If $u=v$ on $U$ then $(dd^cu)^n|_U=(dd^cv)^n|_U$.
\end{theorem}

To prove our general result we only need the following weaker particular case of Theorem \ref{thm1.1} where $m=1$:

\begin{theorem}\label{thm1.2}
Let $\Omega$ be a bounded hyperconvex domain in $\CC^n$ and
$\varphi$ a plurusubharmonic function
on $\Omega$. Let $U=\{\varphi> c\}$ where $c$ is a real constant.
Assume that
$u,v \in \cal E(\Omega)$. If $u=v$ on $U$ then $(dd^cu)^n|_{U}=(dd^cv)^n|_{U}$.
\end{theorem}

\begin{proof}
We adapt the proof of Theorem \ref{thm1.1} given in \cite{HH} to our case.
For each $k\in \NN^*$, set $u_k=\max\{u,-k\}$ and $v_k=\max\{v,-k\}$.
Then $u_k,v_k\in \PSH(\Omega)\cap L_{loc}^\infty(\Omega)$, $u_k\searrow u$, $v_k\searrow v$ as $k\to \infty$.
On the other hand, from the equality $u=v$  on $U$, it follows that $u_k=v_k$ on $U$ and hence
$$(dd^cu_k)^n|_U=(dd^cv_k)^n|_U$$
according to Corollary 4.3 in \cite{BT}. Thus
$$(\max(\varphi,c)-c)(dd^cu_k)^n=(\max(\varphi, c)-c)(dd^cv_k)^n$$
for all $k\geq 1$.

Letting $k\to +\infty$, we deduce by Corollary 3.2 in \cite{H} that
$$(\max(\varphi,c)-c)(dd^cu)^n=(\max(\varphi,c)-c))(dd^cv)^n.$$
for all $c\in \QQ$. But
$$\max(\varphi,c)-c=\varphi-c> 0$$
on $\{\varphi >c\}$, so that $(dd^cu)^n=(dd^cv)^n$ on $\{\varphi > c\}$, and the desired conclusion follows.
\end{proof}

\begin{remark}
Because every domain $\Omega\subset \CC^n$ can be written as $\Omega=\bigcup_{j\in \NN}\Omega_j$
where $\Omega_j$, $j=1,2,...$, are open balls, and hence hyperconvex, it follows that
Theorem \ref{thm1.1} is true for any domain $\Omega\subset \CC^n$ and $u,v\in \cal D$, the
domain of definition of Monge-Amp\`ere operator.
\end{remark}

For a set $E\subset \Omega$, we define
$$u_E(z)=\sup \{v(z): v\in \PSH_{-}(\Omega), \ v\le -1 \text{ on } E\}$$
and $u_E^*$ the upper semicontinuous regularization of $u_E$, that is the function
defined on $\Omega$ by
$$u_E^*(z)=\limsup_{\zeta\to z}u_E(\zeta)$$
for every $z\in \Omega$, (the relative extremal function of $E$, see \cite[p. 158]{Kl}).

\begin{prop}\label{prop1.1}
Let $\Omega$ a domain of $\CC^n$. If a subset $A$ of $\Omega$ is plurithin at $z\in \CC^n$, then there is an open
set $\sigma_z$ containing $z$ such that
$u_{(A\setminus \{z\})\cap \sigma_z}^*(z)>-1$.
\end{prop}

\begin{proof}
The result is obvious if $z\notin \overline A$. Suppose that $z\in \overline A$ and that
$A$ is plurithin at $z$. According to \cite[Proposition 2.2]{BT},
there is a psh function $u\in \PSH_-(\Omega)$ such that
$$\limsup_{\zeta\to z, \zeta\in A, \zeta\ne z}u(\zeta)< u(z).$$
Hence, there is an open set $\sigma_z\subset \Omega$ containing $z$, and a real $\alpha< 0$ such that
$$u(\zeta)\le \alpha <u(z)$$
for every $\zeta\in (A\setminus \{z\})\cap \sigma_z$. The function $v=\frac{-1}{\alpha}u$
is psh $\le 0$ on $\Omega$ and satisfies $v(\zeta)\leq -1$
for every $\zeta \in (A\setminus \{z\})\cap \sigma_z$, so that
$$-1<v(z)\le u_{(A\setminus \{z\})\cap \sigma_z}^*(z).$$
\end{proof}

Now we can state the main result of the present article:

\begin{theorem}\label{thm3.3}
Let $\Omega$ be an hyperconvex open subset of $\CC^n$ and $U$ an $\cal F$-open subset of $\Omega$
that is Borel measurable, and
let $u$, $v$ be two psh functions in the Cegrell class $\cal E(\Omega)$. If $u=v$
on $U$, then $(dd^cu)^n|_U=(dd^cv)^n|_U$.
\end{theorem}

\begin{proof}
Let $(\omega_j)$ be a base of the Euclidean topology on $\CC^n$ formed by open
balls relative to the usual Euclidean norm on $\CC^n$
and let $z\in U$. The set $\complement U=\Omega\setminus U$ being thin at $z$, then, according to Proposition
\ref{prop1.1}, there is an integer $j_z$ such that $u^*_{(\complement U) \cap \omega_{j_z}}(z)>-1$.
Denoting by $\bar u^*_{{\complement U}\cap \omega_{j_z}}$ the function
defined in the same manner as $u^*_{({\complement U})\cap\omega_{j_z}}$ with $\Omega$
replaced by $\omega_{j_z}$, we obviously have
$$z\in V_z:=\{\bar u^*_{({\complement U})\cap \omega_{j_z}}>-1\}\subset \omega_{j_z}$$
because
$$-1<u^*_{\complement U\cap \omega_{j_z}}(z)\leq \bar u^*_{{\complement U}\cap \omega_{j_z}}(z).$$
It is clear that  $V_z$ is a plurifinely open set.
On the set $\complement U\cap \omega_{j_z}$ we have $\bar u^*_{{\complement U}\cap \omega_{j_z}}=-1$ q.e., and hence
$V_z=U\cap V_z\cup F_z$ for some pluripolar set $F_z\subset \omega_{j_z}$.
On the other hand, we have
$$\bigcup_{z\in U} V_z=\bigcup_{z\in U}\{\bar u^*_{({\complement U})\cap \omega_{j_z}}>-1\}
=\bigcup_{j\in J}\{\bar u^*_{({\complement U})\cap \omega_j}>-1\},$$
where $J=\{j_z: z\in U\} (\subset \NN)$. For each $j\in J$, there is a point $z_j\in U$ such that
$j=j_{z_j}$, so that
$$U\subset \bigcup_{z\in U}V_z=\bigcup_{j\in J}\{\bar u^*_{({\complement U})\cap \omega_{j_{z_j}}}>-1\}
=\bigcup_{j\in J}V_{z_j}.$$

The restrictions of
the psh functions $u$ and $v$ to $\omega_{j_{z_j}}$ are equal on $V_{z_j}\setminus F_{z_j}\subset U$ since they are equal on $U$,
and therefore they are equal on $V_{z_j}$ by plurifine  continuity.
It then follows that $(dd^cu)^n|_{V_{z_j}}=(dd^cv)^n|_{V_{z_j}}$ according to Theorem \ref{thm1.1}
applied with $\Omega=\omega_{j_{z_j}}$. From this we infer that $(dd^cu)^n|_{\bigcup_j V_{z_j}}=(dd^c)^n|_{\bigcup_j V_{z_j}}$ and hence
$(dd^cu)^n|_U=(dd^c)^n|_U$ because $U\subset \bigcup_jV_{z_j}$,
and the proof is complete.
\end{proof}

\begin{remark}
In the proof of Theorem \ref{thm3.3}, we also proved that any plurifinely open
subset of a domain $\Omega\subset \CC^n$ is of the form
$\bigcup U_j\setminus P$, where $U_j$, $j=1,2,...$, is of the form $U_j=\{\zeta \in B_j: \varphi_j>-1\}$,
$\varphi_j$ is a plurisubharmonic function on an open ball $B_j$, $j=1,2,...$, and $P$ a pluripolar set.
\end{remark}

\begin{cor}\label{cor3.4}
Let $\Omega$ be an open subset of $\CC^n$, $U$ a
plurifinely open subset of $\Omega$ that is Borel measurable, and $u,v\in \cal D$, the domain of definition
of the Monge-Amp\`ere operator on $\Omega$. If $u=v$ on $U$, then $(dd^cu)^n|_U=(dd^cv)^n|_U$.
\end{cor}

\begin{proof}
We can find open balls $B_j\subset \CC^n$, $j=1, 2, ...,$ such that $\Omega=\bigcup B_j$.
According to Theorem \ref{thm3.3} we have $(dd^cu)^n|_{(U\cap B_j)}=(dd^cv)^n|_{(U\cap B_j)}$ for every $j\ge 1$
and therefore $(dd^cu)^n|_U=(dd^cv)^n|_U$.
\end{proof}

\section{A generalzation of Theorem 3.3}

Let us first recall the following result of Hai and Hiep:

\begin{prop}[{\cite[Proposition 4.1]{HH}}]\label{prop4.1}
Let $\Omega$ be an open subset of $\CC^n$, $u$ a psh function on $\Omega$ that is bounded
near the boundary of $\Omega$ and $T$ a closed positive current on $\Omega$
of bidegree $(p,p)$ ($p<n$), then
the current $dd^c(u T)$ has locally finite mass  in $\Omega$.
\end{prop}

Proposition \ref{prop4.1} allows us to define the current
$dd^cu\wedge T$ on $\Omega$ by putting
$$dd^cu\wedge T=dd^c(uT).$$

The following proposition is a generalization of Proposition 4.4 from \cite{HH}:

\begin{prop}\label{prop4.3}
Let $\Omega\subset \CC^n$ be an open set, $T$ a closed positive current of bidegree
$(p,p)$ on $\Omega$ $(p<n)$ and $\omega$ an open subset of $\Omega$. Assume that
$u_j$, $j=1,2,...$, and $u$ are plurisubharmonic functions bounded near the boundary of $\Omega$.
If $u_j\searrow u$ then the currents $h(\varphi_1, ..., \varphi_m)(dd^cu_j\wedge T|_{\omega}) \to h(\varphi_1, ..., \varphi_m)(dd^cu\wedge T|_{\omega})$
weakly (on $\omega$) for all $\varphi_1,...,\varphi_m\in \PSH \cap L_{loc}^\infty(\omega)$ and $h\in \cal C(\RR^m)$.
\end{prop}

\begin{proof}
Let $B=B(z,r)$ and $B'=B(z,r')$ be  open balls of $\CC^n$ such that $\overline {B'}\subset B\subset \overline B\subset \omega$
and let $\psi$ the psh function on $\Omega$ defined by
$\psi(\zeta)=|\zeta|^2-r_0^2$, where $r_0$ is chosen such that $r'<r_0<r$. Since the functions $\varphi_1,...,\varphi_m$ are bounded on $\overline{B'}$
and $\psi(\zeta)>0$ outside of $B(z,r_0)$ we can find
a constant $A>0$ such that, for every $j=1,\ldots,$,  $\max\{\varphi_j,A\psi\}=\varphi_j$ on $\overline{B'}$ and $\max\{\varphi_j,A\psi\}=A\psi$
outside a compact neighborhood of $\overline B$ in $\omega$. By the sheaf  property of plurisubharmonic functions,
the function defined by $\psi_j=\max\{\varphi_j,A\psi\}$ on $B$ and $\psi_j=A\psi$ on $\Omega\setminus B$ is
a locally bounded psh on $\Omega$. It follows from Proposition 4.4 from \cite{HH} that $h(\psi_1,...,\psi_m)(dd^cu_j\wedge T)|_{B'}\to
h(\psi_1,...,\psi_m)(dd^cu\wedge T)|_{B'}$ weakly  (as currents on $B'$), and hence $h(\varphi_1,...,\varphi_m)(dd^cu_j\wedge T|_\omega)|_{B'}\to
h(\varphi_1,...,\varphi_m)(dd^cu\wedge T|_\omega)|_{B'}$ weakly on $B'$ for we have $h(\psi_1,...,\psi_m)dd^cu_j\wedge T)|_{B'}=h(\varphi_1,...,\varphi_m)(dd^cu_j\wedge T)_\omega)|_{B'}$ for every $j$, and $(h(\psi_1,...,\psi_m)dd^cu\wedge T)|_{B'}=(h(\varphi_1,...,\varphi_m)(dd^cu\wedge T)|_\omega)|_{B'}$. It follows that $h(\varphi_1, ..., \varphi_m)(dd^cu_j\wedge T)|_{\omega} \to h(\varphi_1, ..., \varphi_m)(dd^cu\wedge T)|_{\omega}$
weakly on $B'$.
The property to be proved being local, hence the proof is complete.
\end{proof}

\begin{prop}[{\cite[Proposition 4.6]{NP}}] \label{prop4.5}
Let $\Omega\subset \CC^n$ be an open set, $T$ a closed positive current of bidegree
$(n-1,n-1)$ on $\Omega$ and $\omega$ an open subset of $\Omega$. Let $\varphi, \psi\in \PSH(\omega)$
and $u,v\in \PSH\cap L_{loc}^\infty(\Omega)$ such that $u=v$ on $\cal O=\{\varphi >\psi\}$. Then
$$ dd^cu\wedge T|_{\cal O}=dd^cv\wedge T|_{\cal O}.$$
\end{prop}

As a consequence of Proposition \ref{prop4.5} and the proof of Theorem \ref{thm3.3} we have the following

\begin{prop}
Let $\Omega\subset \CC^n$ be an open set and $T$ a closed positive current of bidegree $(n-1,n-1)$.
Assume that $u,v\in \PSH(\Omega)\cap L_{loc}^\infty(\Omega)$ are such that $u=v$ on a Borel measurable $\cal F$-open set $\cal O\subset \Omega$,
then
$$dd^cu\wedge T|_{\cal O}=dd^cv\wedge T|_{\cal O}.$$
\end{prop}

\begin{proof}
In the proof of Theorem \ref{thm3.3} we have seen that there is a sequence
$(\omega_j)$ of open subsets of $\CC^n$ such that $\cal O$ is a subset of the union $\bigcup V_j$ of $\cal F$-open sets  $V_j$, $j=1,2,\ldots$, of
$\Omega$ all of the forms $V_j=\{\psi_j>-1\}$, where $\psi_j\in \PSH(\omega_j)$, and such that $u=v$ on each of them.
According to Proposition \ref{prop4.5} we have $ dd^cu\wedge T|_{V_j}=dd^cv\wedge T|_{V_j}$ for every $j=1,2,\ldots$, and
thus $ dd^cu\wedge T|_{\cal O}=dd^cv\wedge T|_{\cal O}.$
\end{proof}

Now we can state the following generalization of Theorem 1.2 from \cite{HH}:

\begin{theorem}\label{thm4.6}
Let $\Omega$ an open set in $\CC^n$, $\omega$ an open subset of $\Omega$,
$T$ a closed positive current on $\Omega$ and $u, v\in \PSH(\Omega)$
bounded near the boundary of $\Omega$. If
$u=v$ on a Borel measurable plurifinely open set $U\subset \Omega$ then
$$dd^cu\wedge T|_U=dd^cv\wedge T|_U.$$
\end{theorem}

\begin{proof}
Let $u, v\in \PSH(\Omega)$
be bounded near the boundary of $\Omega$ and equal
on a Borel measurable  plurifinely open set $U\subset \Omega$. We have seen in the proof of Theorem \ref{thm3.3}
that there is a sequence of open set $\omega_j\subset \Omega$, a sequence of plurifinely open sets $V_j \subset \omega_j$
of the forms $V_j=\{\psi_j>-1\}(\subset \omega_j)$ for some sequence of functions $\psi_j\in \PSH(\omega_j)$ such that
for every $j$, $-1\leq \psi_j\leq 0$, $u=v$ on $V_j$, and $U\subset \bigcup_j V_j$.
For every integer $k$ let $u_k=\max(u,-k)\in L_{loc}^\infty(\Omega)$ and $v_k=\max(v,-k)\in \L_{loc}^\infty(\Omega)$, then $u_k\searrow u$ and $v_k\searrow v$. Since
$u_k=v_k$ on $V_j$, we have, according to Proposition \ref{prop4.5}, $dd^cu_k\wedge T|_{V_j}=dd^cv_k\wedge T|_{V_j}$ and
hence
$$(\psi_j+1)(dd^cu_k\wedge T)|\omega_j=(\psi_j+1)(dd^cv_k\wedge T)|\omega_j$$
for any $j$ and $k$. By letting $k\to +\infty$ we therefore have
for every $j$ $$(\psi_j+1)(dd^cu\wedge T)|_{\omega_j}=(\psi_j+1)(dd^cv\wedge T)|_{\omega_j}$$
according to Proposition \ref{prop4.3} so that
$$(dd^cu\wedge T)|_{V_j}=(dd^cv\wedge T)|_{V_j}$$
because $0<\psi_j+1\leq 1$ on $V_j$ and $V_j$ is a Borel subset of $\omega_j$. We conclude that
$$(dd^cu\wedge T)|_{\bigcup_jV_j}=(dd^cv\wedge T)|_{\bigcup_jV_j},$$ whence
$$(dd^cu\wedge T)|_U=(dd^cv\wedge T)|_U.$$

\end{proof}

\end{document}